\documentclass[a5paper, 10pt]{article}
\usepackage{cite, amsmath, amssymb, booktabs, lastpage, graphicx}
\usepackage[hidelinks]{hyperref}
\usepackage{xcolor}
\usepackage{geometry}
\usepackage{setspace}

\usepackage{amsthm}

\theoremstyle{plain}
\newtheorem{theorem}{Theorem}
\newtheorem{conjecture}{Conjecture}
\newtheorem{proposition}[theorem]{Proposition}
\newtheorem{problem}[theorem]{Problem}
\newtheorem{lemma}[theorem]{Lemma}
\newtheorem{corollary}{Corollary}

\theoremstyle{remark}

\theoremstyle{definition}

\usepackage{graphicx}
\makeatletter
\renewcommand{\maketitle}{
	\begin{center}

		{\Large\bfseries \@title} \par
		\vspace{5mm}
		\baselineskip=0.2in
		{\large\bfseries \@author}\par
		\vspace{1mm}
		{\it \@address} \par
		{\small\tt \@email} \par
		\vspace{3mm}
		{\small (Received \@date)} \par
	\end{center}
	\vspace{3mm}
}
\newcommand{\address}[1]{\def\@address{#1}}
\newcommand{\email}[1]{\def\@email{#1}}

\makeatother

\newcommand{\acknowledgment}[1]{\vspace{5mm}\singlespacing
	{\noindent\textbf{\textit{Acknowledgement\/}:} #1}
}

\usepackage{fancyhdr}

\usepackage[margin=1cm,%
			font=footnotesize,%
			format=hang,%
			labelsep=period,%
			labelfont=bf]{caption}

\title{On a Conjecture Concerning the Complementary Second Zagreb Index}
\author{Hicham Saber$^a$, Tariq Alraqad$^a$, Akbar Ali$^{a,}$\footnote{Corresponding author.}, Abdulaziz M. Alanazi$^b$, Zahid Raza$^c$}

\address{$^a$Department of Mathematics,  College of Science,\\ University of Ha\!'il, Ha\!'il, Saudi Arabia\\
$^b$Department of Mathematics, Faculty of Sciences,\\
University of Tabuk, P.O. Box 741 Tabuk, Saudi Arabia\\    $^c$Department of Mathematics, College of
Sciences,\\ University of Sharjah, Sharjah, UAE
    }

\email{hi.saber@uoh.edu.sa, t.alraqad@uoh.edu.sa, akbarali.maths@gmail.com, am.alenezi@ut.edu.sa, zraza@sharjah.ac.ae}

\date{XXXX}

\begin{document}

\maketitle

\begin{abstract}

The
complementary second Zagreb index of a graph $G$ is defined as
$cM_2(G)=\sum_{uv\in E(G)}|(d_u(G))^2-(d_v(G))^2|$, where $d_u(G)$ denotes the degree of a vertex $u$ in $G$ and $E(G)$ represents the edge set of $G$. Let $G^*$ be a graph having the maximum value of $cM_2$ among all connected graphs of order $n$. Furtula and Oz [MATCH Commun. Math. Comput. Chem. 93 (2025) 247--263] conjectured that $G^*$ is the join $K_k+\overline{K}_{n-k}$ of the complete graph $K_k$ of order $k$ and the complement $\overline{K}_{n-k}$ of the complete graph $K_{n-k}$ such that the inequality $k<\lceil n/2 \rceil$ holds. We prove that (i) the maximum degree of $G^*$ is $n-1$ and (ii) no two vertices of minimum degree in $G^*$ are adjacent; both of these results support the aforementioned conjecture. We also prove that the number of vertices of maximum degree in $G^*$, say $k$, is at most  $-\frac{2}{3}n+\frac{3}{2}+\frac{1}{6}\sqrt{52n^2-132n+81}$, which implies that $k<5352n/10000$. Furthermore, we establish results that support the conjecture under consideration for certain bidegreed and tridegreed graphs. In the aforesaid paper, it was also mentioned that determining the $k$ as a function of
the $n$ is far from being an easy task; we obtain the values of $k$ for $5\le n\le 149$ in the case of certain bidegreed graphs by using computer software and found that the resulting sequence of the values of $k$ does not exist in
``The On-Line Encyclopedia of Integer Sequences'' (an online database of integer sequences).

\end{abstract}

\onehalfspacing

\section{Introduction}

Molecular descriptors offer a fundamental tool for the virtual screening of molecule libraries and for predicting the physicochemical characteristics of molecules \cite{Basak-book}.
According to Todeschini and Consonni \cite{Todeschini-20-book}, ``the final result of a logic and mathematical procedure which transforms
chemical information encoded within a symbolic representation of a molecule into a useful number or
the result of some standardized experiment'' is known as a molecular descriptor. Those molecular descriptors that are defined via the graph of a molecule are usually referred to as topological indices in
chemical graph theory \cite{Wagner-18,Trina-book}. (The graph-theoretical and chemical graph-theoretical terminology used in this study not defined here can be found in \cite{Bondy-book,Chartrand-16} and \cite{Wagner-18,Trina-book}, respectively.)
The readers interested in the chemical applications of topological indices are referred to the recent publications
\cite{Desmecht-JCIM-24,Leite-WIREs24}.

Those topological indices that are defined using the vertex degrees of graphs are commonly known as degree-based topological indices \cite{Gutman-13}. Gutman \cite{Gutman-21-MATCH} introduced a geometric approach to devise degree-based topological indices. This approach was extended in \cite{Gutman-24}. Recently, Furtula and Oz \cite{Furtula-MATCH-25} presented a new way of contemplating the concept of ``geometrical'' degree-based topological indices. Using the definition of the second Zagreb index (see \cite{Gutman-75,Borovicanin-17-MATCH}), they \cite{Furtula-MATCH-25} defined the so-called
``complementary second Zagreb index'' (CSZ index, in short). The CSZ index of a graph $G$ is defined \cite{Furtula-MATCH-25} by
$$cM_2(G)=\sum_{uv\in E(G)}|(d_u(G))^2-(d_v(G))^2|,$$
where $d_u(G)$ denotes the degree of a vertex $u$ in $G$ and $E(G)$ represents the set of edges of $G$. As observed in \cite{Furtula-MATCH-25}, the CSZ index was not introduced there for the first time. This index was
proposed independently in several recent papers under different names, including the nano Zagreb index, the F-minus index, the first Sombor index, and the modified Albertson index (see \cite{Furtula-MATCH-25}).

Let $H_1$ and $H_2$ be two graphs with disjoint vertex sets. The join of $H_1$ and $H_2$ is denoted by $H_1+H_2$ and is defined as a graph with the vertex set $V(H_1)\cup V(H_2)$ and edge set $$E(H_1)\cup E(H_2)\cup\{h_1h_2: h_1\in V(H_1), h_2\in V(H_2)\}.$$
Throughout this paper, whenever we use the notation or terminology concerning the join of two graphs, it would be understood that the graphs under consideration have disjoint vertex sets. The complete graph of order $n$ is denoted by $K_n$. The complement of a graph $G$ is represented as $\overline{G}$. In \cite{Furtula-MATCH-25}, the following conjecture was posed:

\begin{conjecture}\label{conj}
If $G^*$ is a graph having the maximum value of $cM_2$ among all connected graphs of order $n$ then $G^*$ is isomorphic to $K_k+\overline{K}_{n-k}$ for some $k$ satisfying $k<\lceil n/2 \rceil$, where $n\ge5$ and the graph $K_k+\overline{K}_{n-k}$ is shown in Figure \ref{Fig-01}.
\end{conjecture}

\vspace{-4mm}
\begin{figure}[!ht]
 \centering
\includegraphics[width=0.25\textwidth]{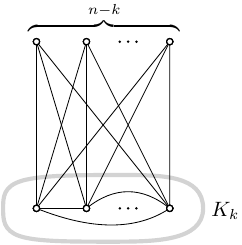}
   \caption{The graph $K_k+\overline{K}_{n-k}$.}
    \label{Fig-01}
     \end{figure}

In \cite{Furtula-MATCH-25}, the following problem was also posed:

\begin{problem}\label{prob}
Assuming that Conjecture \ref{conj} is true,
determine $k$ as a function of
$n$.
\end{problem}

In \cite{Furtula-MATCH-25}, it was mentioned that solving Problem \ref{prob} is ``far from being an easy task''. In the present paper, we provide some results concerning the solutions to Conjecture \ref{conj} and Problem \ref{prob}. More precisely,
we prove that (i) the maximum degree of the graph $G^*$ defined in Conjecture \ref{conj} is $n-1$ and (ii) no two vertices of minimum degree in $G^*$ are adjacent; both of these results support the Conjecture \ref{conj}. We also prove that the number of vertices of maximum degree in $G^*$ is at most $$-\frac{2}{3}n+\frac{3}{2}+\frac{1}{6}\sqrt{52n^2-132n+81}.$$ Furthermore, we establish results that support Conjecture \ref{conj} for certain bidegreed and tridegreed graphs. Concerning Problem \ref{prob}, we obtain the values of $k$ for $5\le n\le 149$ in the case of certain bidegreed graphs by using computer software and found that the resulting sequence of the values of $k$ does not exist in
``The On-Line Encyclopedia of Integer Sequences''  \cite{OLEIS}.

\section{Results}

By an $n$-order graph, we mean a graph of order $n$. For a graph $G$ and a vertex $u\in V(G)$, let $N_G(u)$ denote the set of vertices adjacent to $u$.

\begin{proposition}\label{l01}
If $G$ is a graph having the maximum value of $cM_2$ among connected $n$-order graphs, $n\ge4$, then the maximum degree of $G$ is $n-1$.
\end{proposition}

\begin{proof}
Let $\Delta$ be the maximum degree of $G$. We assume that $\Delta<n-1$ and seek a contradiction. We pick a vertex $v\in V(G)$ of degree $\Delta$ and choose another vertex $u\in V(G)$ that is not adjacent to $v$. We also define $N_1:=\{y\in N_G(u):d_y(G)> d_u(G)\}$. Let $G'$ be the graph obtained from $G$ by adding the edge $uv$. In the rest of the proof, we take $d_s=d_s(G)$ for every vertex $s\in V(G')=V(G)$.
Then we have
\begin{align}
cM_2(G')-cM_2(G)=
&\sum_{w\in N_G(v)}\left[((\Delta+1)^2-d_w^2)-(\Delta^2-d_w^2)\right]\nonumber\\[2mm]
&+\sum_{x\in N_G(u)\setminus N_1}\left[((d_u+1)^2-d_x^2)-(d_u^2-d_x^2)\right]\nonumber\\[2mm]
&+\sum_{y\in N_1}\left[(d_y^2-(d_u+1)^2)-(d_y^2-d_u^2)\right]\nonumber\\[2mm]
&+(\Delta+1)^2-(d_u+1)^2\nonumber\\[2mm]
=&\Delta(2\Delta+1)+(d_u-|N_1|)(2d_u+1)\nonumber\\[2mm]
&-|N_1|(2d_u+1)+(\Delta+1)^2-(d_u+1)^2\nonumber\\[2mm]
=&\Delta(2\Delta+1)+(d_u-2|N_1|)(2d_u+1)\nonumber\\[2mm]
&+(\Delta+1)^2-(d_u+1)^2.\label{Eq-001}
\end{align}
If $d_u=\Delta$, then $|N_1|=0$, and hence Equation \eqref{Eq-001} yields
$$cM_2(G')-cM_2(G)=2\Delta(2\Delta+1)>0,$$
a contradiction to the maximality of $cM_2(G)$.
Next, we assume that the inequality $d_u<\Delta$ holds. Since $|N_1|\leq d_u$, we get
$$d_u-2|N_1|\geq -|N_1|\geq-d_u> -\Delta.$$ Hence, Equation \eqref{Eq-001} yields
\begin{align*}
cM_2(G')-cM_2(G)>&\,\Delta(2\Delta+1)-\Delta(2d_u+1)+(\Delta+1)^2-(d_u+1)^2\\
=&\,2\Delta(\Delta-d_u)+(\Delta+1)^2-(d_u+1)^2>0,
\end{align*}
again a contradiction.
Therefore, the maximum degree of $G$ is $n-1$.
\end{proof}

\begin{proposition}\label{l02}
If $G$ is a graph having the maximum value of $cM_2$ among all connected $n$-order graphs with $n\ge4$, then no two vertices of minimum degree in $G$ are adjacent.
\end{proposition}
\begin{proof}
Let $\delta$ be the minimum degree of $G$. Contrarily, we suppose that $x,y\in V(G)$ are adjacent vertices of degree $\delta$. Since $n\ge4$, we have $\delta\ge2$. If $G'$ is the graph obtained from $G$ by removing the edge $xy$, then we have
\begin{align*}
cM_2(G')-cM_2(G)=&\sum_{u\in N(x)\setminus \{y\}}\left[(d_u^2-(d_x-1)^2)-(d_u^2-d_x^2)\right]\\[2mm]
&+\sum_{v\in N(y)\setminus \{x\}}\left[(d_v^2-(d_y-1)^2)-(d_v^2-d_y^2)\right]\\[2mm]
=&2(\delta-1)(2\delta-1)>0,
\end{align*}
a contradiction to the maximality of $cM_2(G)$, where we used the notation $d_s=d_s(G)$ for every vertex $s\in V(G')=V(G)$. Thus, no two vertices of minimum degree in $G$ are adjacent.
\end{proof}

We remark here that Propositions \ref{l01} and \ref{l02} support Conjecture \ref{conj}. Next, we derive an upper bound on the number of vertices of maximum degree in a graph having the maximum value of $cM_2$ among all connected $n$-order graphs. For a graph $G$, denote by $M(G)$ the set of those vertices of $G$ that have the maximum degree.

\begin{proposition}\label{l03}
If $G$ is a graph having the maximum value of $cM_2$ among all connected $n$-order graphs with $n\ge4$, then
$$|M(G)|\leq -\frac{2}{3}n+\frac{3}{2}+\frac{1}{6}\sqrt{52n^2-132n+81}.$$
\end{proposition}

\begin{proof}
Let $k=|M(G)|$. By Proposition \ref{l01}, we have $d_u(G)=n-1$ for every $u\in M(G)$. We choose an edge $xy\in E(G)$ in such a way that the vertex $x$ has the maximum degree and the vertex $y$ has the minimum degree. Let $G'$ denote the graph obtained from $G$ by removing the edge $xy$. In the rest of the proof, we take $d_s=d_s(G)$ for every vertex $s\in V(G')=V(G)$. Now, we have
\begin{align}
0\geq &\,cM_2(G')-cM_2(G)\nonumber\\
=&\sum_{u\in M(G)\setminus \{x\}}\left[(n-1)^2-(n-2)^2\right]\nonumber\\
&+\sum_{v\in N_G(x)\setminus (M(G)\cup \{y\})}\left[((n-2)^2-d_v^2)-((n-1)^2-d_v^2)\right]\nonumber\\
&+\sum_{w\in N_G(y)\setminus \{x\}}\left[((d_w)^2-(d_y-1)^2)-(d_w^2-d_y^2)\right]-((n-1)^2-d_y^2)\nonumber\\
=&(k-1)(2n-3)+(n-k-1)(3-2n)+(d_y-1)(2d_y-1)\nonumber\\
&+d_y^2-(n-1)^2\nonumber\\%(n-k-1) -> (n-k-2)
=&(2k-n)(2n-3)+(d_y-1)(2d_y-1)+d_y^2-(n-1)^2.\label{Eq-002}
\end{align}
Since $d_y\ge |M(G)|=k$, the expression given on the right most of \eqref{Eq-002} is greater than or equal to
\[
(2k-n)(2n-3)+(k-1)(2k-1)+k^2-(n-1)^2,
\]
which is equal to $3k^2+(4n-9)k-3n^2+5n$.
Therefore, from \eqref{Eq-002}, we have $$3k^2+(4n-9)k-3n^2+5n\leq 0,$$
which implies that the product of the expressions
$$k+\frac{2}{3}n-\frac{3}{2}+\frac{1}{6}\sqrt{52n^2-132n+81}$$
and
$$k+\frac{2}{3}n-\frac{3}{2}-\frac{1}{6}\sqrt{52n^2-132n+81}$$
is less than or equal to $0$.
Therefore, we have $$k+\frac{2}{3}n-\frac{3}{2}-\frac{1}{6}\sqrt{52n^2-132n+81}\le 0,$$ which gives the desired bound on $|M(G)|$.
\end{proof}

Since, for $n\ge4$, the inequality
$$-\frac{2}{3}n+\frac{3}{2}+\frac{1}{6}\sqrt{52n^2-132n+81}<\frac{5352}{10000}n$$
holds, from Proposition \ref{l03}, the next result follows.

\begin{corollary}
If $G$ is a graph having the maximum value of $cM_2$ among all connected $n$-order graphs with $n\ge4$, then
\begin{equation*}\label{eq-bound}
 |M(G)|<\frac{5352}{10000}n.
\end{equation*}
\end{corollary}

The primary motivation of establishing Proposition \ref{l03} and its corollary is the inequality $|M(G)|<\lceil n/2 \rceil$ given in Conjecture \ref{conj}.

The degree set of a graph $G$ is the set of all distinct degrees of the vertices of $G$. By an $\ell$-degreed graph, we mean a graph having a degree set consisting of exactly $\ell$ elements. If $\ell=1,2,$ or $3,$ then the corresponding $\ell$-degreed graph is called a regular graph, a bidegreed graph, or a tridegreed graph, respectively.

Since the extremal graph mentioned in Conjecture \ref{conj} is a bidegreed graph and because the maximum degree of this extremal graph is $n-1$ by Proposition \ref{l03}, we next prove Conjecture \ref{conj} for the class of $n$-order bidegreed connected graphs with maximum degree $n-1$. Also, in the next result, we assume that $n\ge11$ because the required extremal graphs are already known for $n\le 10$; see \cite{Furtula-MATCH-25}.

\begin{proposition}\label{l05}
Let $G$ be a connected bidegreed $n$-order graph of maximum degree $n-1$, with $n\ge11$ and $|M(G)|=k$. Then
\begin{equation}\label{eq-bideg-new}
cM_2(G)\leq k(n-k)\left((n-1)^2-k^2\,\right),
\end{equation}
with equality if and only if $G=K_k+\overline{K}_{n-k}$. Also, the inequality
\begin{equation}\label{eq-new-prop4}
cM_2(G)\le \epsilon\, n^2(1-\epsilon)\left((n-1)^2-\epsilon^2n^2\,\right)
\end{equation}
holds for some $\epsilon$ lying between $372/1000$ and $392/1000$.
\end{proposition}

\begin{proof}
If $\delta$ is the minimum degree of $G$, then $n-2\ge\delta\geq k\ge 1$ and hence we have
\begin{align}\label{eq-004}
cM_2(G)=k(n-k)((n-1)^2-\delta^2)\leq k(n-k)\left((n-1)^2-k^2\,\right),
\end{align}
where the right equality in \eqref{eq-004} holds if and only if $\delta=k$; this completes the proof of \eqref{eq-bideg-new}. Next,
we define the function $f$ as $$f(x,y)=x(y-x)((y-1)^2-x^2)$$ for real variables $x$ and $y$ satisfying $y\ge x+2 \ge3$ and $y\ge11$. Then $$\frac{\partial f}{\partial x}=4x^3-3yx^2-2(y-1)^2x+y(y-1)^2.$$
If $1 \le x \le \frac{373}{1000} y$ and $y\ge11$, then we note that $\frac{\partial f}{\partial x}>0$. Also, for the case when $\frac{391}{1000} y \le x \le y-2$ and $y\ge11$, then we have $\frac{\partial f}{\partial x}<0$. Hence, for every fix $y\geq 11$, $f(x,y)$ attain its maximum value over the interval $[1,y-2]$ at some $x$ lying between $\frac{372}{1000}y$ and $\frac{392}{1000}y$. Now, \eqref{eq-new-prop4} follows from \eqref{eq-004}.
\end{proof}

Since the difference $\frac{4}{10}n-\frac{3}{10}n$ is strictly greater than $1$ for $n\ge11$, from the proof of Proposition \ref{l05} the next result follows.

\begin{corollary}\label{cor-Kr}
If $G$ is a graph having the maximum value of $cM_2$ over the class $\{K_k+\overline{K}_{n-k}: 1\le k\le n-2, \, n\ge11\}$, then  $\frac{3}{10}n < k < \frac{4}{10}n$.
\end{corollary}

As mentioned in the introduction section, solving Problem \ref{prob} is ``far from being an easy task'' \cite{Furtula-MATCH-25}. By restricting ourselves to $n$-order connected bidegreed graphs of maximum degree $n-1$, and by using computer software, we find the values of $k$ (see Table \ref{Tab-1}) for $5\le n\le 149$.
The sequence consisting of the values of $k$ given in Table \ref{Tab-1} (that is, {\small$2,2,3,3,3,4,4,4,5,5,6,6,6,7,7,8,8,8,9,9,10,\!10,\!10,\!11,\!11,\!12,\!12,\!12,\!13,\!13,\!13,\dots$}) does not exist in
``The On-Line Encyclopedia of Integer Sequences'' \cite{OLEIS}.

\begin{table}[ht!]
\begin{center}
\caption{The values of $k$ (the number of vertices of degree $n-1$) in $n$-order connected bidegreed graphs with maximum $cM_2$ for $5\le n \le 149$.}\label{Tab-1}
\begin{tabular}{|c|c|c|c|c|c|c|c|c|c|c|c|c|c|}
\hline
$n$ & $k$    &  & $n$ & $k$    &  & $n$ & $k$    &  & $n$ & $k$    &  & $n$ & $k$    \\ \hline
5   & \textbf{2}  &  & 34  & 13          &  & 63  & 24          &  & 92  & \textbf{36} &  & 121 & 47          \\ \cline{1-2} \cline{4-5} \cline{7-8} \cline{10-11} \cline{13-14}
6   & \textbf{2}  &  & 35  & 13          &  & 64  & \textbf{25} &  & 93  & \textbf{36} &  & 122 & 47          \\ \cline{1-2} \cline{4-5} \cline{7-8} \cline{10-11} \cline{13-14}
7   & 3           &  & 36  & \textbf{14} &  & 65  & \textbf{25} &  & 94  & 37          &  & 123 & \textbf{48} \\ \cline{1-2} \cline{4-5} \cline{7-8} \cline{10-11} \cline{13-14}
8   & 3           &  & 37  & \textbf{14} &  & 66  & 26          &  & 95  & 37          &  & 124 & \textbf{48} \\ \cline{1-2} \cline{4-5} \cline{7-8} \cline{10-11} \cline{13-14}
9   & 3           &  & 38  & 15          &  & 67  & 26          &  & 96  & 37          &  & 125 & 49          \\ \cline{1-2} \cline{4-5} \cline{7-8} \cline{10-11} \cline{13-14}
10  & 4           &  & 39  & 15          &  & 68  & 26          &  & 97  & 38          &  & 126 & 49          \\ \cline{1-2} \cline{4-5} \cline{7-8} \cline{10-11} \cline{13-14}
11  & 4           &  & 40  & 15          &  & 69  & \textbf{27} &  & 98  & 38          &  & 127 & 49          \\ \cline{1-2} \cline{4-5} \cline{7-8} \cline{10-11} \cline{13-14}
12  & 4           &  & 41  & \textbf{16} &  & 70  & \textbf{27} &  & 99  & 38          &  & 128 & \textbf{50} \\ \cline{1-2} \cline{4-5} \cline{7-8} \cline{10-11} \cline{13-14}
13  & \textbf{5}  &  & 42  & \textbf{16} &  & 71  & 28          &  & 100 & \textbf{39} &  & 129 & \textbf{50} \\ \cline{1-2} \cline{4-5} \cline{7-8} \cline{10-11} \cline{13-14}
14  & \textbf{5}  &  & 43  & 17          &  & 72  & 28          &  & 101 & \textbf{39} &  & 130 & 51          \\ \cline{1-2} \cline{4-5} \cline{7-8} \cline{10-11} \cline{13-14}
15  & 6           &  & 44  & 17          &  & 73  & 28          &  & 102 & 40          &  & 131 & 51          \\ \cline{1-2} \cline{4-5} \cline{7-8} \cline{10-11} \cline{13-14}
16  & 6           &  & 45  & 17          &  & 74  & 29          &  & 103 & 40          &  & 132 & 51          \\ \cline{1-2} \cline{4-5} \cline{7-8} \cline{10-11} \cline{13-14}
17  & 6           &  & 46  & \textbf{18} &  & 75  & 29          &  & 104 & 40          &  & 133 & \textbf{52} \\ \cline{1-2} \cline{4-5} \cline{7-8} \cline{10-11} \cline{13-14}
18  & \textbf{7}  &  & 47  & \textbf{18} &  & 76  & 29          &  & 105 & \textbf{41} &  & 134 & \textbf{52} \\ \cline{1-2} \cline{4-5} \cline{7-8} \cline{10-11} \cline{13-14}
19  & \textbf{7}  &  & 48  & 19          &  & 77  & \textbf{30} &  & 106 & \textbf{41} &  & 135 & 53          \\ \cline{1-2} \cline{4-5} \cline{7-8} \cline{10-11} \cline{13-14}
20  & 8           &  & 49  & 19          &  & 78  & \textbf{30} &  & 107 & 42          &  & 136 & 53          \\ \cline{1-2} \cline{4-5} \cline{7-8} \cline{10-11} \cline{13-14}
21  & 8           &  & 50  & 19          &  & 79  & 31          &  & 108 & 42          &  & 137 & 53          \\ \cline{1-2} \cline{4-5} \cline{7-8} \cline{10-11} \cline{13-14}
22  & 8           &  & 51  & \textbf{20} &  & 80  & 31          &  & 109 & 42          &  & 138 & 54          \\ \cline{1-2} \cline{4-5} \cline{7-8} \cline{10-11} \cline{13-14}
23  & \textbf{9}  &  & 52  & \textbf{20} &  & 81  & 31          &  & 110 & \textbf{43} &  & 139 & 54          \\ \cline{1-2} \cline{4-5} \cline{7-8} \cline{10-11} \cline{13-14}
24  & \textbf{9}  &  & 53  & 21          &  & 82  & \textbf{32} &  & 111 & \textbf{43} &  & 140 & 54          \\ \cline{1-2} \cline{4-5} \cline{7-8} \cline{10-11} \cline{13-14}
25  & 10          &  & 54  & 21          &  & 83  & \textbf{32} &  & 112 & 44          &  & 141 & \textbf{55} \\ \cline{1-2} \cline{4-5} \cline{7-8} \cline{10-11} \cline{13-14}
26  & 10          &  & 55  & 21          &  & 84  & 33          &  & 113 & 44          &  & 142 & \textbf{55} \\ \cline{1-2} \cline{4-5} \cline{7-8} \cline{10-11} \cline{13-14}
27  & 10          &  & 56  & 22          &  & 85  & 33          &  & 114 & 44          &  & 143 & 56          \\ \cline{1-2} \cline{4-5} \cline{7-8} \cline{10-11} \cline{13-14}
28  & \textbf{11} &  & 57  & 22          &  & 86  & 33          &  & 115 & 45          &  & 144 & 56          \\ \cline{1-2} \cline{4-5} \cline{7-8} \cline{10-11} \cline{13-14}
29  & \textbf{11} &  & 58  & 22          &  & 87  & \textbf{34} &  & 116 & 45          &  & 145 & 56          \\ \cline{1-2} \cline{4-5} \cline{7-8} \cline{10-11} \cline{13-14}
30  & 12          &  & 59  & \textbf{23} &  & 88  & \textbf{34} &  & 117 & 45          &  & 146 & \textbf{57} \\ \cline{1-2} \cline{4-5} \cline{7-8} \cline{10-11} \cline{13-14}
31  & 12          &  & 60  & \textbf{23} &  & 89  & 35          &  & 118 & \textbf{46} &  & 147 & \textbf{57} \\ \cline{1-2} \cline{4-5} \cline{7-8} \cline{10-11} \cline{13-14}
32  & 12          &  & 61  & 24          &  & 90  & 35          &  & 119 & \textbf{46} &  & 148 & 58          \\ \cline{1-2} \cline{4-5} \cline{7-8} \cline{10-11} \cline{13-14}
33  & 13          &  & 62  & 24          &  & 91  & 35          &  & 120 & 47          &  & 149 & 58          \\ \cline{1-2} \cline{4-5} \cline{7-8} \cline{10-11} \cline{13-14}
\end{tabular}
\end{center}
\end{table}

Next, we prove a lemma, which is needed to prove Conjecture \ref{conj} for $n$-order tridegreed connected graphs of maximum degree $n-1$.

\begin{lemma}\label{lem-5}
Let $x,y,z$ be real numbers such that $0\leq x\leq y\leq z$. Then for any nonnegative real numbers $a,b,c,$ the following inequality holds:
$$a(z^2-y^2)+b(z^2-x^2)+c(y^2-x^2)\leq \left(b+\max\{a,c\}\right)(z^2-x^2).$$
\end{lemma}
\begin{proof}
Note that
\begin{align*}
a(z^2-y^2)+b(z^2-x^2)+c(y^2-x^2)
\le&\max\{a,c\}(z^2-y^2)+b(z^2-x^2)\\
&+\max\{a,c\}(y^2-x^2)\\
     =&(b+\max\{a,c\})(z^2-x^2).
\end{align*}
\end{proof}

Let $G$ be an $\ell$-degreed connected graph with degree set  $\{d_1,d_2,\cdots,d_\ell\}$ such that $d_1<d_2<\cdots<d_\ell$ and $\ell\ge2$. For every $i\in \{1,2,\cdots,\ell\}$, we define $V_i=\{u\in V(G): d_u(G)=d_i\}$ and $\nu_i=|V_i|$. Also, for every $i\in \{1,\cdots,t-1\}$ and  $j\in \{i+1,\cdots,t\}$, we define
$$a_{i,j}=\big|\{uv\in E(G): u\in V_i, v\in V_j\}\big|.$$
Then
\begin{equation}\label{eq1} cM_2(G)=\sum_{i=1}^{t-1}\sum_{j=i+1}^{t}a_{i,j}(d_j^2-d_i^2)\end{equation}

\begin{proposition}\label{l06}
If $G$ is a tridegreed connected $n$-order graph with degree set $\{d_1,d_2,d_3\}$ such that $d_1<d_2<d_3=n-1$ and $n\ge11$, then $$cM_2(G)< cM_2(K_{t}+\overline{K}_{n-t})$$
for some $t$ lying between $\frac{3}{10}n$ and $\frac{4}{10}n$.
\end{proposition}
\begin{proof} By using \eqref{eq1}, we have
\begin{equation}\label{Trideg-eq}
  cM_2(G)=\nu_3\nu_2((n-1)^2-d_2^2)+\nu_3\nu_1((n-1)^2-d_1^2)+a_{1,2}(d_2^2-d_1^2).
\end{equation}
{\bf Case 1.} $\nu_3\nu_2>a_{1,2}$.\\
By utilizing Lemma \ref{lem-5} in Equation \eqref{Trideg-eq}, we obtain
\begin{equation}\label{Trideg-Eq-2}
cM_2(G)< \nu_3(\nu_1+\nu_2)((n-1)^2-d_1^2).
\end{equation}
Since every vertex of the set $V_3$ is adjacent to all vertices of $V_1$, we obtain $\nu_3=|V_3|\leq d_1$. Additionally, we have $\nu_1+\nu_2=n-\nu_3$. Thus, from \eqref{Trideg-Eq-2} it follows that $$cM_2(G)< \nu_3(n-\nu_3)((n-1)^2-\nu_3^2)=cM_2(K_{\nu_3}+\overline{K}_{n-\nu_3}),$$
which implies the desired inequality because of Corollary \ref{cor-Kr}.\\[2mm]
{\bf Case 2.} $\nu_3\nu_2 \le a_{1,2}$.\\
By using Lemma \ref{lem-5} in Equation \eqref{Trideg-eq}, we obtain
\begin{equation}\label{Trideg-Eq-3}
 cM_2(G)\leq (a_{1,2}+\nu_3\nu_1)((n-1)^2-d_1^2).
\end{equation}
{\bf Case 2.1.} $\nu_2+\nu_3\leq d_1$.\\
We note that $a_{1,2}\leq \nu_1\nu_2$. If any of the inequalities $\nu_3\nu_2 \le a_{1,2}$, $\nu_2+\nu_3\leq d_1$, and $a_{1,2}\leq \nu_1\nu_2$, is strict, then from  \eqref{Trideg-eq} and \eqref{Trideg-Eq-3} it follows that
\begin{align}\label{Trideg-Eq-3-a}
cM_2(G)&< \nu_1(\nu_2+\nu_3)((n-1)^2-(\nu_2+\nu_3)^2)\\
&=cM_2(K_{\nu_2+\nu_3}+\overline{K}_{n-\nu_2-\nu_3}).\nonumber
\end{align}
If $\nu_3\nu_2 = a_{1,2}$, $\nu_2+\nu_3= d_1$, and $a_{1,2}= \nu_1\nu_2$, then $\nu_1=\nu_3$ and hence $\nu_2+\nu_3> n/2$ (for otherwise, the inequality $\nu_2+\nu_3\le n/2$ gives $\nu_1 \ge n/2$, which gives $\nu_1+\nu_3\ge n/2+n/2$, a contradiction); therefore, from \eqref{Trideg-eq}, \eqref{Trideg-Eq-3}, and Corollary \ref{cor-Kr}, it follows that
\begin{align*}\label{Trideg-Eq-3-b}
cM_2(G)&= \nu_1(\nu_2+\nu_3)((n-1)^2-(\nu_2+\nu_3)^2)\\
&=cM_2(K_{\nu_2+\nu_3}+\overline{K}_{n-\nu_2-\nu_3})<cM_2(K_{t}+\overline{K}_{n-t}),\nonumber
\end{align*}
for some $t$ lying between $\frac{3}{10}n$ and $\frac{4}{10}n$.\\[2mm]
{\bf Case 2.2.}  $\nu_2+\nu_3 > d_1$.\\
In this case, we have $\nu_1=n-(\nu_2+\nu_3)<n-d_1$. {Now, for every $x\in V_1$, we have $d_x=|N_G(x)\cap V_1|+|N_G(x)\cap V_2|+|N_G(x)\cap V_3|$. Since $V_3\subseteq N_G(x)$, and $d_x=d_1$, we get $|N_G(x)\cap V_2|\leq d_1-\nu_3$; by summing this over all elements $x$ of $V_1$, we obtain $a_{1,2}\leq \nu_1(d_1-\nu_3)$}. Hence, from \eqref{Trideg-Eq-3} we obtain
\begin{equation}\label{Trideg-Eq-3d}
 cM_2(G)< (n-d_1)d_1((n-1)^2-d_1^2)=cM_2(K_{d_1}+\overline{K}_{n-d_1}).
\end{equation}
Now, \eqref{Trideg-Eq-3d} implies the desired inequality because of Corollary \ref{cor-Kr}.
\end{proof}

\acknowledgment
{This work is supported by the Scientific Research Deanship, University of Ha\!'il, Ha\!'il, Saudi Arabia, through project number RG-24\,059.}

% References should be ordered alphabetically.
\singlespacing


\begin{thebibliography}{00}

\bibitem{Basak-book} S. C. Basak (Ed.),
{\it Mathematical Descriptors of Molecules and Biomolecules:
Applications in Chemistry, Drug Design, Chemical Toxicology, and Computational Biology}, Springer, Cham, 2024.

\bibitem{Bondy-book} J. A. Bondy, U. S. R. Murty, {\it Graph Theory}, Springer, Heidelberg, 2008.

\bibitem{Borovicanin-17-MATCH} B. Borovi\'canin, K. C. Das, B. Furtula, I. Gutman, Bounds for Zagreb indices,
{\it MATCH Commun. Math. Comput. Chem.} {\bf78} (2017) 17--100.

\bibitem{Chartrand-16} G. Chartrand, L. Lesniak, P. Zhang, {\it Graphs} \& {\it Digraphs}, CRC Press, 2016.

\bibitem{Desmecht-JCIM-24} D. Desmecht, V. Dubois, Correlation of the molecular cross-sectional area of organic monofunctional compounds with topological descriptors, {\it J. Chem. Inf. Model.} {\bf64} (2024) 3248--3259.

\bibitem{Furtula-MATCH-25} B. Furtula, M. S. Oz, Complementary topological
indices, {\it MATCH Commun. Math. Comput. Chem.} {\bf93} (2025) 247--263.

\bibitem{Gutman-13} I. Gutman, Degree-based topological indices, {\it Croat. Chem. Acta} {\bf86}
(2013) 351--361.

\bibitem{Gutman-21-MATCH} I. Gutman, Geometric approach to degree-based topological indices: Sombor indices, {\it MATCH Commun. Math. Comput. Chem.} {\bf86} (2021) 11--16.

\bibitem{Gutman-24} I. Gutman, B. Furtula, M. S. Oz, Geometric approach to vertex-degree-based topological indices -- Elliptic Sombor index, theory and
application, {\it Int. J. Quantum Chem.} {\bf124} (2024) \#e27346.

\bibitem{Gutman-75} I. Gutman, B. Ru\v{s}\v{c}i\'c, N. Trinajsti\'c, C. F. Wilcox, Graph theory and molecular orbitals, XII. Acyclic polyenes, {\it J. Chem. Phys.} {\bf62} (1975) 3399--3405.

\bibitem{Leite-WIREs24} L. S. G. Leite, S. Banerjee, Y. Wei, J. Elowitt, A. E. Clark, Modern chemical graph theory, {\it WIREs Comput. Mol. Sci.} {\bf14} (2024) \#e1729.

%\bibitem{Liu-JMS-23} H. Liu, Mathematical and chemical properties of geometry-based invariants and its applications, {\it J. Mol. Struct.} {\bf1291} (2023) \#136060.

%\bibitem{Tang-MATCH-23} Z. Tang, Q. Li, H. Deng, Trees with extremal values of the Sombor-index-like graph invariants, {\it MATCH Commun. Math. Comput. Chem.} {\bf90} (2023) 203--222.

\bibitem{OLEIS} {\it The On-Line Encyclopedia of Integer Sequences}, available online at \url{https://oeis.org}.

\bibitem{Todeschini-20-book} R. Todeschini, V. Consonni, {\it Handbook of Molecular Descriptors}, Wiley-VCH,
Weinheim, 2000.


\bibitem{Trina-book} N. Trinajsti\'c, {\it Chemical Graph Theory},  CRC Press, Boca Raton, 1992.


\bibitem{Wagner-18} S. Wagner, H. Wang, {\it Introduction to Chemical Graph Theory},
        CRC Press, Boca Raton, 2018.
%\bibitem{Yousaf-UM-20} S. Yousaf, A. A. Bhatti, A. Ali, A note on the modified Albertson index, {\it Util. Math.} {\bf117} (2020) 139--146.

\end{thebibliography}
\end{document}